\documentclass[12pt]{amsart}
\usepackage{amssymb,amsmath,amsthm,enumerate}

\DeclareMathOperator{\sign}{sign}
\DeclareMathOperator{\Ran}{Ran}
\DeclareMathOperator{\Dom}{Dom}
\DeclareMathOperator{\Ker}{Ker}
\DeclareMathOperator{\Tr}{Tr}

\DeclareMathOperator{\spec}{spec}
\DeclareMathOperator{\rank}{rank}
\DeclareMathOperator{\supp}{supp}
\DeclareMathOperator*{\slim}{s-lim}
\DeclareMathOperator*{\LIM}{l.i.m.}

\renewcommand\Im{\hbox{{\rm Im}}\,}
\renewcommand\Re{\hbox{{\rm Re}}\,}
\newcommand{\abs}[1]{\lvert#1\rvert}
\newcommand{\Abs}[1]{\left\lvert#1\right\rvert}
\newcommand{\norm}[1]{\lVert#1\rVert}

\newcommand{\bbR}{{\mathbb R}}
\newcommand{\bbC}{{\mathbb C}}
\newcommand{\bbN}{{\mathbb N}}

\newcommand{\calH}{{\mathcal H}}

\newcommand{\calF}{\mathcal{F}}
\newcommand{\calN}{\mathcal{N}}

\newcommand{\calZ}{\mathcal{Z}}
\newcommand{\calY}{\mathcal{Y}}
\newcommand{\calP}{\mathcal{P}}

\numberwithin{equation}{section}

\theoremstyle{plain}
\newtheorem{theorem}{\bf Theorem}[section]
\newtheorem{lemma}[theorem]{\bf Lemma}
\newtheorem{proposition}[theorem]{\bf Proposition}
\newtheorem{assumption}[theorem]{\bf Assumption}
\newtheorem{corollary}[theorem]{\bf Corollary}

\theoremstyle{definition}

\theoremstyle{remark}
\newtheorem*{remark*}{\bf Remark}

\newcommand{\wt}{\widetilde}
\newcommand{\eps}{\varepsilon}

\begin{document}

\title{Scattering matrix and functions of self-adjoint operators}

\author{Alexander Pushnitski}
\address{Department of Mathematics,
King's College London, 
Strand, London, WC2R~2LS, U.K.}
\email{alexander.pushnitski@kcl.ac.uk}



\begin{abstract}
In the scattering theory framework, we consider a pair 
of operators $H_0$, $H$. 
For a continuous function $\varphi$
vanishing at infinity,  we set 
$\varphi_\delta(\cdot)=\varphi(\cdot/\delta)$
and study the spectrum of the difference 
$\varphi_\delta(H-\lambda)-\varphi_\delta(H_0-\lambda)$ 
for $\delta\to0$. 
We prove that if $\lambda$ is in the absolutely continuous spectrum of $H_0$ and $H$, 
then  the spectrum of this difference  
converges to a set that can be explicitly 
described in terms of (i) the eigenvalues of the scattering matrix $S(\lambda)$ for 
the pair $H_0$, $H$  and (ii) the singular values of the Hankel operator $H_\varphi$
with the symbol $\varphi$. 
\end{abstract}

\maketitle

\section{Introduction}\label{sec.a}

 \subsection{Informal description of the main result}\label{sec.a2}
Let $H_0$ and $H$ be self-adjoint operators in a Hilbert space $\calH$.
We assume that $H_0$ and $H$ are semibounded from below
and $H=H_0+V$ as a quadratic form sum, where 
$V$ is a self-adjont operator in $\calH$ such that 
$\abs{V}=\sqrt{V^2}$ is $H_0$--form compact. 
Under these assumptions, the difference
of resolvents 
$(H-zI)^{-1}-(H_0-zI)^{-1}$ is compact for every 
non-real $z$ and the essential spectra 
of $H_0$ and $H$ coincide.

Let $C_0(\bbR)$ be the set of all continuous functions $\varphi:\bbR\to\bbR$
such that $\varphi(x)\to0$ as $\abs{x}\to\infty$.  
For  $\varphi\in C_0(\bbR)$ 
and $\delta>0$, we denote $\varphi_\delta(x)=\varphi(x/\delta)$. 
Fix $\lambda\in\bbR$ and consider the difference 
\begin{equation}
A(\delta)=\varphi_\delta(H-\lambda)-\varphi_\delta(H_0-\lambda).
\label{a1}
\end{equation}
Under our assumptions it is easy to see 
that the operator $A(\delta)$ is compact. 
In this paper, we study the spectrum of $A(\delta)$ for 
$\delta\to+0$. 

In Section~\ref{sec.b}
we make some assumptions typical for smooth scattering theory.
These assumptions
ensure that the (local) wave operators corresponding to the pair $H_0$, $H$
and some interval $\Delta$ in the absolutely continuous spectrum of $H_0$ 
exist and are complete. Thus, the scattering matrix $S(\lambda)$ 
is well defined for $\lambda\in\Delta$. 
For $\lambda\in\Delta$ 
we describe the limiting behaviour of the spectrum of $A(\delta)$ 
as $\delta\to+0$ explicitly in terms of 
\begin{enumerate}[(i)]
\item
the eigenvalues of the scattering matrix $S(\lambda)$;
\item 
the singular values 
of the (compact) Hankel operator $H_\varphi$ with the symbol $\varphi$.
\end{enumerate}
To give a general flavour of our result, let us consider the case
$\varphi(x)=1/(1+x^2)$. This case turns out to be particularly simple
as the operator  $H_\varphi$ has rank one. 
We prove that 
the spectrum of $A(\delta)$ converges (in some precise sense to be
specified in Section~\ref{sec.b}) to the set
\begin{equation}
\{
\pm\tfrac14\abs{s_n(\lambda)-1}: 
s_n(\lambda)\in\spec(S(\lambda))
\}.
\label{a2}
\end{equation}

We note that
the link between the properties of the difference 
$\varphi(H)-\varphi(H_0)$ and the theory of Hankel operators
was first exhibited in the work \cite{Peller1} by V.~Peller. 
The question discussed in this paper gives another example 
of this link.

 \subsection{Connection to the Birman-Kre\u{\i}n formula}\label{sec.a1}

In  \cite{Krein}, M.~G.~Kre\u{\i}n has proved that
under some assumptions of the trace class type on the pair $H_0$ and $H$,
for all sufficiently smooth functions $\varphi\in C_0(\bbR)$ 
the operator $\varphi(H)-\varphi(H_0)$ belongs to the trace class and 
\begin{equation}
\Tr(\varphi(H)-\varphi(H_0))=\int_{-\infty}^\infty \varphi'(t) \xi(t) dt,
\label{aa1}
\end{equation}
where the function $\xi$ is known as the spectral shift function. 
The relation between the spectral shift function and 
the   scattering matrix $S(\lambda)$
for the pair $H_0$, $H$ was found later in the paper 
\cite{BK} by M.~Sh.~Birman and M.~G.~Kre\u{\i}n
(see also the previous work \cite{Lifshits,BuslaevFaddeev,Buslaev}):
\begin{equation}
\det S(\lambda)=e^{-2\pi i\xi(\lambda)}
\label{aa2}
\end{equation}
for almost every $\lambda$ in the  absolutely continuous   spectrum of $H_0$.

Obviously, the l.h.s. of \eqref{aa1} is the sum of the eigenvalues of 
$\varphi(H)-\varphi(H_0)$ and the l.h.s. of \eqref{aa2} is the product 
of the eigenvalues of $S(\lambda)$. 
Thus, \eqref{aa1} and \eqref{aa2} relate the spectrum of $\varphi(H)-\varphi(H_0)$
to the spectrum of the scattering matrix. The main result of this paper
gives an affirmative answer to the following 

\textbf{Question:} \emph{are there any other relationships between 
the spectrum of $\varphi(H)-\varphi(H_0)$ for smooth $\varphi$ 
and the spectrum of the
scattering matrix $S(\lambda)$ for the pair $H_0$, $H$?}

For discontinuous functions $\varphi$ the 
operator $A(\delta)$ may fail to be compact; see \cite[Section~6]{Krein}
and \cite{KM}.
In this case the essential spectrum of $A(\delta)$ 
can be explicitly described in terms of 
the spectrum of the scattering matrix; see \cite{Push,PushYaf}.
This fact is closely related to the subject of this work;
it gives another relationship between the spectra of $\varphi(H)-\varphi(H_0)$
and $S(\lambda)$.

\subsection{Acknowledgements}
The author is grateful to Yu.~Safarov and N.~Filonov for a number of useful remarks 
on the text of the paper.

\section{Main result}\label{sec.b}

\subsection{Notation and assumptions}\label{sec.b1}
For a self-adjoint operator $A$ and a Borel set $\Lambda\subset\bbR$, we denote 
by $E_A(\Lambda)$ the spectral projection 
of $A$  corresponding to $\Lambda$ and let 
$N_A(\Lambda)=\rank E_A(\Lambda)$. 
If $\Lambda=(a,b)$, we write $E_A(a,b)$, $N_A(a,b)$  rather than $E_A((a,b))$, 
$N_A((a,b))$  in order to make the formulas more readable.

We assume that $H_0$ is 
a semi-bounded from below self-adjoint operator in a Hilbert space
$\calH$, and $V$ is another operator in $\calH$ which is considered
as the perturbation of $H_0$. 
It will be convenient to represent $V$ in a factorised form: 
$V=GV_0 G$, where $G=\abs{V}^{1/2}$ and $V_0=\sign(V)$. 
We assume  that 
for any $\gamma<\inf\spec(H_0)$ one has 
\begin{equation}
\Dom (H_0-\gamma I)^{1/2} \subset \Dom G 
\quad  \text{ and } \quad
G(H_0-\gamma I)^{-1/2} \text{ is compact.}
\label{b1}
\end{equation}
It follows that $V$ is $H_0$-form compact, and therefore 
we can define the self-adjoint operator $H$ 
corresponding to  the form sum $H_0+V$
(see the ``KLMN Theorem''  \cite[Theorem~X.17]{RS2}). 

For $\Im z\not=0$, we set $R(z)=(H-zI)^{-1}$, $R_0(z)=(H_0-zI)^{-1}$. 
Let us define the ``sandwiched resolvent'' $T(z)$ formally by 
$$
T(z)=GR_0(z)G^*, 
\quad
\Im z\not=0;
$$
more precisely, this means
\begin{equation}
T(z)=(G(H_0-\gamma I)^{-1/2}) (H_0-\gamma I)R_0(z)(G(H_0-\gamma I)^{-1/2})^*
\label{b5}
\end{equation}
for any $\gamma <\inf\spec(H_0)$.
By \eqref{b1}, the operator $T(z)$ is  compact.

We fix a compact interval $\Delta\subset\bbR$ and assume that 
the spectrum of $H_0$ in $\Delta$ is purely absolutely continuous
with a constant multiplicity $N\leq\infty$. 
More explicitly, we assume that for some auxiliary Hilbert space $\calN$, 
$\dim\calN=N$, there exists a 
unitary operator $\calF$  from $\Ran E_{H_0}(\Delta)$ to 
$L^2(\Delta,\calN)$,  such that $\calF$
diagonalizes $H_0$: if $f\in\Ran E_{H_0}(\Delta)$ then 
\begin{equation}
(\calF  H_0 f)(\lambda)
=
\lambda (\calF f)(\lambda), \quad \lambda\in\Delta.
\label{b2}
\end{equation}

Next, we make an assumption typical for smooth scattering theory; 
in the terminology of \cite{Yafaev}, we
assume that $G$ is strongly $H_0$-smooth 
on $\Delta$ with some exponent $\alpha\in (0, 1]$.
This means that the operator 
\[
G_\Delta\overset{\rm def}{=}GE_{H_0}(\Delta) : \Ran E_{H_0}(\Delta) \to \calH
\] 
satisfies 
\begin{equation}
(\calF G_\Delta^*\psi)(\lambda)
=
Z(\lambda)\psi, 
\quad 
\forall \psi\in\calH,
\quad
\lambda\in\Delta,
\label{b3}
\end{equation}
where $Z=Z(\lambda):\calH\to\calN$
is a family of compact operators obeying
\begin{equation}
\norm{Z(\lambda)}\leq C,
\quad
\norm{Z(\lambda)-Z(\lambda')}\leq C\abs{\lambda-\lambda'}^\alpha,
\quad
\lambda, \lambda'\in\Delta.
\label{b4}
\end{equation}
Note that the notion of strong smoothness is not unitary invariant,
as it depends on the choice of the map $\calF$. 
It follows from \eqref{b3} that the  operator  
$G_\Delta$ acts according to the formula
\begin{equation}
G_\Delta f=
\int_{\Delta}Z(\lambda)^* F(\lambda) d\lambda, 
\quad 
F=\calF f.
\label{b4a}
\end{equation}

Let us summarize our assumptions:

\begin{assumption}\label{as1}
\begin{enumerate}[\rm (A)]
\item
$H=H_0+V$ (as a form sum), where $V=GV_0G$ satisfies \eqref{b1}.
\item
$H_0$ has a purely absolutely continuous spectrum with multiplicity $N$ on the interval $\Delta$.
\item
$G=\abs{V}^{1/2}$ is strongly $H_0$-smooth on $\Delta$, i.e.  \eqref{b3}, \eqref{b4} hold true.
\end{enumerate}
\end{assumption}

\subsection{Scattering theory}\label{sec.b2}

Recall that for a Borel set 
${\Lambda}\subset{\mathbb R}$, the (local) wave operators are introduced by the relation
\[
W_{\pm} (H,H_{0};\Lambda)=\slim_{t\to \pm \infty}e^{iHt}e^{-iH_{0}t}E_{H_0}(\Lambda) P_{H_0}^{(a)},
\]
provided these strong limits exist. 
Here and in what follows we denote by  $P_{H_0}^{(a)}$ the orthogonal projection
onto the absolutely continuous subspace   of $H_{0}$. 
If the wave operators are complete, i.e.  if the relations  
\[
\Ran W_{+} (H,H_{0};\Lambda)
=
\Ran W_{-} (H,H_{0};\Lambda)
=
\Ran \big(E_H (\Lambda) P^{(a)}_H\big)
\]
hold true, then 
the (local) scattering operator is defined as
\[
\mathbf{S}=\mathbf{S}(H,H_0;\Lambda)= W_+(H,H_0;\Lambda)^* W_-(H,H_0;\Lambda).
\]
The scattering operator $\mathbf{S}$
commutes with $H_{0}$ and is unitary on the subspace $\Ran( E_{H_0}(\Lambda)P_{H_0}^{(a)})$.

Below the interior of $\Delta$ is denoted by $ {\rm int}(\Delta)$. 
We need the following well-known results (see e.g. \cite[Section~4.4]{Yafaev}).

\begin{proposition}\label{pr1}
Let Assumption~\ref{as1} hold. Then:
\begin{enumerate}[\rm (i)]
\item
The operator-valued function $T(z)$   
defined by \eqref{b5}
is uniformly  H\"older continuous for $\Re z\in {\rm int}(\Delta)$, $\Im z> 0$; 
in particular, the limits $T(\lambda+i0)$ exist in the operator norm and 
are H\"older continuous in $\lambda\in {\rm int}(\Delta)$.
Let $\Omega\subset{\rm int}(\Delta)$ be the set 
where the equation 
$$
f+T(\lambda+ i0)V_0f=0
$$ 
has no non-trivial solutions. Then $\Omega$ 
is open and  $\Delta\setminus\Omega$
has the Lebesgue measure zero.
The inverse operator $(I+T(\lambda+i0)V_0)^{-1}$, $\lambda\in\Omega$, 
exists, is bounded and is a H\"older continuous function of  
$\lambda\in\Omega$.
\item
The local wave operators
$W_\pm(H,H_0;\Omega)$ exist and are complete. 
Moreover, the spectrum of $H$ in $\Omega$ 
is purely absolutely continuous.
\end{enumerate}
\end{proposition}

The last statement of  Proposition~\ref{pr1} is usually formulated 
under the additional assumption $\Ker G=\{0\}$. 
Actually, this assumption is not necessary; this is verified    
in Lemma~A.1 of  \cite{PushYaf}.

Since the scattering operator $\mathbf{S}$
commutes with $H_{0}$,
we have a representation
\[
( \calF \mathbf{S}  \calF^* f)(\lambda) =S(\lambda) f(\lambda), 
\quad \text{ a.e. }\lambda\in \Delta,
\]
where the operator  $S(\lambda):\calN\to\calN$  is  called the scattering matrix
for the pair of operators $H_0$, $H$. 
The scattering matrix  is a unitary operator in $\calN$. 
We need the stationary representation for the scattering matrix
(see \cite[Chapter~7]{Yafaev} for the details): 
\begin{equation}
S(\lambda)=I-2\pi i Z(\lambda)V_0(I+T(\lambda+i0)V_0)^{-1}Z(\lambda)^*,
\quad 
 \lambda\in\Omega.
\label{b5a}
\end{equation}
This representation, in particular, implies that $S(\lambda)$ is 
a H\"older continuous function of $\lambda\in\Omega$. 
Since the operator $V_0(I+T(\lambda+i0)V_0)^{-1}$  
is bounded and $Z(\lambda)$ is compact, it follows that the 
operator $S(\lambda)-I$ is  compact.  
Thus, the spectrum of $S(\lambda)$ consists of  eigenvalues 
accumulating possibly only to the point 1. 
All eigenvalues of $S(\lambda)$ distinct from $1$ have finite multiplicities.
We denote by $\{s_n(\lambda)\}_{n=1}^N$ the eigenvalues of 
$S(\lambda)$, enumerated with multiplicities taken into account.

\subsection{Hankel operators}\label{sec.b3}
Recall that the Hardy space  $H^2(\bbC_\pm)\subset L^2(\bbR)$ 
is defined as the class of all functions $f$ analytic in $\bbC_\pm=\{z\in\bbC: \pm \Im z>0\}$ 
and satisfying the estimate
$$
\sup_{y>0}\int_{\bbR} \abs{f(x\pm iy)}^2 dx<\infty.
$$
Let $P_\pm$ be the orthogonal projection
in $L^2(\bbR)$ onto $H^2(\bbC_\pm)$. 
The explicit formula for $P_\pm$ is 
\begin{equation}
(P_\pm f)(x)=\mp\frac1{2\pi i}\LIM_{\eps\to+0}
\int_\bbR \frac{f(y)}{x-y\pm i\eps}dy, 
\label{hp}
\end{equation}
where $\LIM$ denotes the limit in $L^2(\bbR)$.

For $\varphi\in C_0(\bbR)$ we denote by 
$\boldsymbol{\varphi}$ the operator
of multiplication by $\varphi(x)$ in $L^2(\bbR,dx)$
and by  $H_\varphi$ 
the Hankel operator in $L^2(\bbR)$ with the symbol 
$\varphi$: $H_\varphi=P_-\boldsymbol{\varphi} P_+$. 
It is well-known \cite{Peller}
that the assumption $\varphi\in C_0(\bbR)$ implies that 
$H_\varphi$ is compact. We denote by $\{\mu_m(\varphi)\}_{m=1}^\infty$ the 
sequence of singular values of $H_\varphi$ enumerated
in decreasing order with multiplicities taken into account. 

It is easy to check that for $\varphi(x)=1/(1+x^2)$ one has
\begin{equation}
H_\varphi f=-\frac14 v(x)\int_{\bbR}f(y)v(y)dy, 
\quad
v(x)=\frac1{\sqrt{\pi}}\frac1{x-i}. 
\label{b5b}
\end{equation}
Since
$\norm{v}_{L^2(\bbR)}=1$
we see that for this choice of $\varphi$ the singular values of $H_\varphi$ are 
$\mu_1(\varphi)=1/4$ and $\mu_m(\varphi)=0$ for all $m\geq2$.

\subsection{Main result}\label{sec.b4}

Let $\varphi\in C_0(\bbR)$; 
fix $\lambda\in\Omega$ 
(the set $\Omega$ is defined in Proposition~\ref{pr1})
and let $A(\delta)$ be as in \eqref{a1}.
Let us define the set 
\begin{equation}
\sigma_0(\varphi,\lambda)
=
\{\pm
\mu_m(\varphi)\abs{s_n(\lambda)-1}:
n=1,\dots,N,\, m\in\bbN\}\cup\{0\}.
\label{b6}
\end{equation}
As we will see, this set is the limiting spectrum of $A(\delta)$ as $\delta\to0$.
The corresponding eigenvalue counting function is defined as
\begin{equation}
N_0(s)
=\#\{n=1,\dots,N,\, m\in\bbN: \mu_m(\varphi)\abs{s_n(\lambda)-1}>s\}, 
\quad s>0.
\label{b7}
\end{equation}
Our main result is 
\begin{theorem}\label{thm.b1}
Let Assumption~\ref{as1} hold true;
fix $\varphi\in C_0(\bbR)$ 
and $\lambda\in\Omega$. Let $A(\delta)$ be as in \eqref{a1}.
Then for any $s>0$, $s\notin\sigma_0(\varphi,\lambda)$, one has
\begin{equation}
N_{\pm A(\delta)}(s,\infty)\to N_0(s),
\quad \text{ as }\delta\to+0.
\label{b8}
\end{equation}
\end{theorem}
It is easy to translate this into the more explicit language
of eigenvalues. 
We will say that a point $\nu\in\sigma_0(\varphi,\lambda)$, 
$\nu\not=0$, has multiplicity $k\geq1$ in $\sigma_0(\varphi,\lambda)$ if  $\nu$ can be represented
as $\pm \mu_m(\varphi)\abs{s_n(\lambda)-1}$ for $k$ different 
choices of $n$, $m$. For $\nu\notin \sigma_0(\varphi,\lambda)$
we set the multiplicity of $\nu$ to zero. 
The multiplicity of $\nu$ can be alternatively defined as
$k=N_0(\abs{\nu}+0)-N_0(\abs{\nu}-0)$. 
\begin{corollary}\label{cr.b2}
Let $\nu\in\bbR$, $\nu\not=0$ and suppose that the multiplicity of 
$\nu$ in $\sigma_0(\varphi,\lambda)$ is $k\geq0$. 
Then for any sufficiently small $\rho>0$ there exists $\delta=\delta(\rho)$ 
such that for all $\delta'\in(0,\delta]$ the operator $A(\delta')$ has exactly 
$k$ eigenvalues (counting multiplicities) in the interval $(\nu-\rho,\nu+\rho)$. 
\end{corollary}

Applying Theorem~\ref{thm.b1} to  $\varphi(x)=1/(1+x^2)$ 
and recalling \eqref{b5b}, we obtain  a
particularly simple relation:
\begin{equation}
N_{\pm A(\delta)}(s,\infty)\to \#\{n=1,\dots,N:\tfrac14\abs{s_n(\lambda)-1}>s\}, 
\quad \delta\to+0,
\label{b9}
\end{equation}
for all $s>0$ such that $s\not=\frac14\abs{s_n(\lambda)-1}$, $n\in\bbN$. 

Theorem~\ref{thm.b1} is proven in Sections~\ref{sec.c} and \ref{sec.d}.
In Section~\ref{sec.c} we introduce a model operator $A_0(\delta)$ 
(see \eqref{c0}) such that $A_0(\delta)$ is unitarily equivalent to $A_0(1)$
for all $\delta>0$ and the spectrum of $A_0(1)$ is given by $\sigma_0(\varphi,\lambda)$.
After this, we prove, roughly speaking, that 
$\norm{A(\delta)-A_0(\delta)}\to0$ as $\delta\to+0$ (see
Lemma~\ref{lma.c2} for the precise statement); this yields \eqref{b8}.
The representation \eqref{b5a} plays a crucial role in the proof.

\subsection{Applications and extensions}\label{sec.b5}

Let $H_0=-\Delta$ in $L^2(\bbR^d)$ with $d\geq1$.
Application of the Fourier transform shows that 
$H_0$ has a purely absolutely continuous spectrum $[0,\infty)$ 
with multiplicity $N=2$ if $d=1$ and $N=\infty$ if $d\geq2$.

Let $H=H_0+V$, where $V$ 
is the operator of multiplication by a function $V:\bbR^d\to\bbR$
which is assumed to satisfy
\begin{equation}
\abs{V(x)}\leq C(1+\abs{x})^{-\rho},
\quad 
\rho>1.
\label{a4}
\end{equation}
Then   Assumption~\ref{as1} is fulfilled on every 
compact subinterval $\Delta$ of $ (0,\infty)$.
 Moreover,  by a well-known argument involving 
Agmon's ``bootstrap" \cite{Agmon}
and Kato's theorem \cite{Kato} on the absence of positive eigenvalues
of $H$, the operator $I+T(\lambda+i0)V_0$ is invertible
for all $\lambda>0$  and hence $\Omega={\rm int}(\Delta)$.  
Thus, Proposition~\ref{pr1} implies that the wave operators 
$W_{\pm} (H,H_{0})$   exist and are complete 
(this result was first obtained in  \cite{Kato3,Kuroda}).
The scattering matrix $S(\lambda)$  is a $2\times 2$ unitary 
matrix if $d=1$ and a  unitary operator in $L^2(\mathbb S^{d-1})$
if $d\geq2$. 
Theorem~\ref{thm.b1} applies to this situation. 

Similar applications are possible in situations where the diagonalization 
of $H_0$ is known explicitly. For example, 
the perturbed
Schr\"odinger operator with a constant magnetic field in dimension 
three can be considered.

Next, the construction of this paper can easily be extended to 
the case  of operators $H_0$, $H$ which are not lower semi-bounded. 
Here one should follow  \cite{PushYaf} to define the 
sum $H_0+V$  in an appropriate way and to prove that $A(\delta)$ 
is compact for $\varphi\in C_0(\bbR)$. 
The rest of the construction remains the same.

\section{Proof of Theorem~\ref{thm.b1}}\label{sec.c}

\subsection{The model operator}\label{sec.c1}
Below we assume without loss of generality that $\lambda=0$ and let $a>0$ be such 
that the interval $[-a,a]$ belongs to the set $\Omega$. 
For $\delta>0$, 
let us define the model operator $A_0(\delta)$ in 
$L^2(\bbR,\calN)=L^2(\bbR)\otimes\calN$ by 
\begin{equation}
A_0(\delta)=H_{\varphi_\delta}\otimes (S(0)-I)+H_{\varphi_\delta}^*\otimes (S(0)^*-I).
\label{c0}
\end{equation}
By definition,  the operator $A_0(\delta)$ is self-adjoint. 
Since both $H_{\varphi_\delta}$ and $S(0)-I$ are compact, the operator $A_0(\delta)$ is 
also compact. 

For $\delta>0$, let $U(\delta)$ in $L^2(\bbR,\calN)$ be the unitary scaling operator:
$$
(U(\delta) f)(x)=\delta^{-1/2}f(x/\delta).
$$
It is straightforward to see that 
\begin{equation}
A_0(\delta)=U(\delta)A_0(1)U(\delta)^*;
\label{c0a}
\end{equation}
in particular, the spectrum and the eigenvalue counting function 
of $A_0(\delta)$ are independent of $\delta$. 

\begin{lemma}\label{lma.c1}
For all $\delta>0$, 
the spectrum of $A_0(\delta)$ coincides with the set $\sigma_0(\varphi,0)$ 
defined in \eqref{b6}. 
For any $s>0$, $s\notin\sigma_0(\varphi,0)$, one has
$$
N_{\pm A_0(\delta)}(s,\infty)
= 
N_0(s),
$$
with $N_0(s)$ defined in \eqref{b7}.
\end{lemma}
\begin{proof}
By \eqref{c0a}, it suffices to consider the case $\delta=1$. 
Let the eigenvalues of $S(0)$ be $s_n(0)=e^{i\theta_n}$, $\theta_n\in[0,2\pi)$, and let 
$\alpha_n=(\theta_n+\pi)/2$. 
Using the spectral decomposition of $S(0)$, one represents
the operator $A_0(\delta)$ as the orthogonal sum of the operators
\begin{equation}
(s_n(0)-1)H_\varphi+(\overline{s_n(0)}-1)H_\varphi^*
=
\abs{s_n(0)-1}(e^{i\alpha_n} P_-\boldsymbol{\varphi}P_+
+
e^{-i\alpha_n}P_+\boldsymbol{\varphi}P_-)
\label{c0c}
\end{equation}
in $L^2(\bbR)$. With respect to the orthogonal sum decomposition 
$L^2(\bbR)=H^2(\bbC_-)\oplus H^2(\bbC_+)$ we have 
\begin{equation}
e^{i\alpha_n} P_-\boldsymbol{\varphi} P_+
+
e^{-i\alpha_n}P_+\boldsymbol{\varphi} P_-
=
\begin{pmatrix}
0 & e^{i\alpha_n}P_-\boldsymbol{\varphi}P_+
\\
e^{-i\alpha_n}P_+\boldsymbol{\varphi} P_- & 0
\end{pmatrix}.
\label{c0b}
\end{equation}
A simple argument shows that if $M$ is a compact operator
with the singular values $\{\mu_m\}_{m=1}^\infty$, 
then the spectrum of 
$\begin{pmatrix}0&M\\M^*&0\end{pmatrix}$
consists of the eigenvalues $\{\pm\mu_m\}_{m=1}^\infty$. 
Thus, we get that the spectrum of the  operator in the r.h.s. of \eqref{c0b} 
consists of the eigenvalues $\{\pm\mu_m(\varphi)\}_{m=1}^\infty$.
Combining this with \eqref{c0c}, we obtain 
the required statement. 
\end{proof}

\subsection{The strategy of proof}\label{sec.c2}
Let $\Pi_\calN$, $\Pi_\calH$ 
be the restriction operators: 
$$
\Pi_\calN:L^2(\bbR,\calN)\to L^2((-a,a),\calN)
\quad \text{ and }\quad  
\Pi_\calH: L^2(\bbR,\calH)\to L^2((-a,a),\calH).
$$
Recall that $\calF$ is defined in Section~\ref{sec.b1} (see \eqref{b2}); 
we set
$$
\calF_a=\calF E_{H_0}(-a,a):  \calH\to L^2((-a,a),\calN).
$$
Consider the partial isometry 
$$
Q: L^2(\bbR,\calN)\to\calH, 
\quad 
Q=\calF_a^* \Pi_\calN. 
$$
It is clear that $\norm{Q}=1$. 
In Section~\ref{sec.d} we prove 
\begin{lemma}\label{lma.c2}
As $\delta\to+0$, one has 
\begin{align}
\norm{A(\delta)-Q A_0(\delta) Q^*}&\to0,
\label{c1}
\\
\norm{(Q^*Q-I)A_0(\delta)}&\to0,
\label{c3}
\\
\norm{(QQ^*-I)A(\delta)}&\to0.
\label{c2}
\end{align}
\end{lemma}
Given Lemma~\ref{lma.c2}, it is not difficult to complete the
proof of Theorem~\ref{thm.b1}:

\begin{proof}[Proof of Theorem~\ref{thm.b1}]
1. 
First note that from \eqref{c1}--\eqref{c2} it follows that
\begin{equation}
\norm{A_0(\delta)-Q^*A(\delta)Q}\to0, 
\quad \delta\to+0.
\label{c4}
\end{equation}
Indeed, 
$$
A_0(\delta)
=
(I-Q^*Q)A_0(\delta)+Q^*QA_0(\delta)(I-Q^*Q)
+Q^*QA_0(\delta)Q^*Q, 
$$
and therefore, using the fact that $\norm{Q}=1$, 
\begin{multline*}
\norm{A_0(\delta)-Q^*A(\delta)Q}
\\
\leq
\norm{(I-Q^*Q)A_0(\delta)}
+
\norm{Q^*QA_0(\delta)(I-Q^*Q)}
+
\norm{Q^*(Q A_0(\delta) Q^*-A(\delta))Q}
\\
\leq
2\norm{(I-Q^*Q)A_0(\delta)}
+
\norm{Q A_0(\delta) Q^*-A(\delta)}\to0
\end{multline*}
as $\delta\to0$. 

2. 
Fix $s>0$, $s\notin\sigma_0(\varphi,0)$ and let $\eps\in(0,s)$, $\eps<1$.   
Using \eqref{c1}--\eqref{c4}, let us choose $\delta_0>0$ 
such that for all $\delta\in(0,\delta_0)$ we have
\begin{align}
\norm{(Q^*Q-I)A_0(\delta)}&<s\eps, 
\quad&
\norm{(QQ^*-I)A(\delta)}&<s\eps, 
\label{c5}
\\
\norm{A_0(\delta)-Q^*A(\delta)Q}&<\eps,
\quad&
\norm{A(\delta)-QA_0(\delta) Q^*}&<\eps.
\label{c6}
\end{align}
Below we prove that from \eqref{c5}, \eqref{c6} it follows that
\begin{align}
N_{A_0(\delta)}(\tfrac{s-\eps}{1+\eps},\infty)
&\geq
N_{A(\delta)}(s,\infty),
\label{c7}
\\
N_{A(\delta)}(\tfrac{s-\eps}{1+\eps},\infty)
&\geq
N_{A_0(\delta)}(s,\infty)
\label{c8}
\end{align}
for all $\delta<\delta_0$. 
Recall that by Lemma~\ref{lma.c1} the spectrum of $A_0(\delta)$ 
is independent of $\delta$ and coincides with $\sigma_0(\varphi,0)$. 
If $\eps$ is sufficiently small, the interval 
$[\tfrac{s-\eps}{1+\eps},s]$ contains no eigenvalues of $A_0(\delta)$ 
and so the l.h.s. of \eqref{c7} equals $N_{A_0(\delta)}(s,\infty)$. 
Thus, \eqref{c7} is equivalent to 
\begin{equation}
N_{A_0(\delta)}(s,\infty)\geq N_{A(\delta)}(s,\infty).
\label{c8a}
\end{equation}
On the other hand, 
denote $s'=s(1+\eps)+\eps$; then \eqref{c5}, \eqref{c6} hold true
with $s'$ in place of $s$ and therefore by \eqref{c8} we have
$$
N_{A(\delta)}(\tfrac{s'-\eps}{1+\eps},\infty)
\geq
N_{A_0(\delta)}(s',\infty).
$$
This can be rewritten as 
$$
N_{A(\delta)}(s,\infty)
\geq
N_{A_0(\delta)}(s(1+\eps)+\eps,\infty).
$$
Again, if $\eps$ is chosen sufficiently small then
the r.h.s. in the last relation 
equals $N_{A_0(\delta)}(s,\infty)$. Thus, we get 
$N_{A(\delta)}(s,\infty)\geq N_{A_0(\delta)}(s,\infty)$; 
combining this with \eqref{c8a} yields
$$
N_{A(\delta)}(s,\infty)=N_{A_0(\delta)}(s,\infty)
$$
for all sufficiently small $\delta$.
In the same way, one proves that 
$$
N_{-A(\delta)}(s,\infty)=N_{-A_0(\delta)}(s,\infty)
$$
for all $s\notin\sigma_0(\varphi,0)$ and all sufficiently small $\delta$. 
Thus, we arrive at \eqref{b8}.

3. It remains to prove \eqref{c7}, \eqref{c8}. 
We prove \eqref{c7}; the relation \eqref{c8} is proven in the same way. 
Let $f\in \Ran E_{A(\delta)}(s,\infty)$, 
$\norm{f}=1$; then $f=A(\delta)g$, $\norm{g}\leq 1/s$. 
Using \eqref{c5}, we get 
$$
\norm{(QQ^*-I)f}
=
\norm{(QQ^*-I)A(\delta)g}
\leq
s\eps/s
=
\eps.
$$
It follows that 
$$
\abs{((QQ^*-I)f,f)}<\eps, 
$$
or equivalently,
$$
\Abs{\norm{Q^*f}^2-1}<\eps.
$$
Since $\eps<1$, we obtain $\norm{Q^*f}\not=0$, and
also 
\begin{equation}
\norm{Q^*f}^2<1+\eps.
\label{c9}
\end{equation}
Further, by the definition of $f$ we have 
$(A(\delta)f,f)\geq s$
and so,
using \eqref{c6}, we get
$$
(A_0(\delta)Q^*f,Q^*f)
=
(A(\delta)f,f)
+
((QA_0(\delta)Q^*-A(\delta))f,f)
\geq 
s-\eps.
$$
Combining this with \eqref{c9}, we obtain
$$
\frac{(A_0(\delta)Q^*f,Q^*f)}
{\norm{Q^*f}^2}
>
\frac{s-\eps}{1+\eps}.
$$
Since we have already seen that $Q^*f\not=0$,
we get that $Q^*$ maps 
$\Ran E_{A(\delta)}(s,\infty)$ onto a subspace $L$, $\dim L=N_{A(\delta)}(s,\infty)$ 
and for all $h\in L$ we have
$$
(A_0(\delta)h,h)>\frac{s-\eps}{1+\eps}\norm{h}^2.
$$
By the min-max principle, we get \eqref{c7}. 
\end{proof}

\section{Proof of Lemma~\ref{lma.c2}}\label{sec.d}

\subsection{Preliminaries}
Here we prove the relations \eqref{c1}, \eqref{c3} and \eqref{c2}. 
The proof of  \eqref{c3} and \eqref{c2} is very straightforward, 
while the proof of \eqref{c1} requires 
 more detailed analysis. We will repeatedly use the following well-known fact. 
Let $M_n$ be a sequence of bounded operators such that $M_n\to0$ 
strongly as $n\to\infty$. Then for any compact operator $K$, one has
$\norm{M_nK}\to0$ as $n\to\infty$. 
In particular, if we also have $M_n^*\to0$ strongly, then 
$\norm{KM_n}=\norm{M_n^*K^*}\to0$ as $n\to\infty$. 

\subsection{The proof of \eqref{c3} and \eqref{c2} }\label{sec.c3}

Let us prove \eqref{c3}. 
We have 
$$
Q^*Q
=
(\calF_a^*{\Pi_\calN})^*\calF_a^*{\Pi_\calN}
=
{\Pi_\calN}^*\calF_a\calF_a^*{\Pi_\calN}
=
{\Pi_\calN}^*{\Pi_\calN}
=
\boldsymbol{\chi_{(-a,a)}},
$$
where $\boldsymbol{\chi_{(-a,a)}}$ is the 
operator of multiplication by the characteristic function 
of the interval $(-a,a)$ in $L^2(\bbR,\calN)$. 
It follows that 
$$
\norm{(Q^*Q-I)A_0(\delta)}
=
\norm{(\boldsymbol{\chi_{(-a,a)}}-I)U(\delta)A_0(1)U(\delta)^*}.
$$
It is straightforward to see that the operator 
$(\boldsymbol{\chi_{(-a,a)}}-I)U(\delta)$ converges to zero strongly
as $\delta\to+0$. Since $A_0(1)$ is a compact operator, we obtain \eqref{c3}.

Let us prove \eqref{c2}. By the definition of $Q$, we have
$$
QQ^*
=
\calF_a^*{\Pi_\calN} {\Pi_\calN}^*\calF_a
=
\calF_a^*\calF_a
=
E_{H_0}(-a,a).
$$
Thus, we need to prove that 
\begin{equation}
\norm{(E_{H_0}(-a,a)-I)A(\delta)}\to0, 
\quad
\delta\to0.
\label{c9a}
\end{equation}
Let $\zeta\in C(\bbR)$ be such that $\zeta(x)=1$ for $\abs{x}\geq a$ 
and $\zeta(x)=0$ for $\abs{x}\leq a/2$. 
Clearly, it suffices to prove that 
\begin{equation}
\norm{\zeta(H_0)A(\delta)}\to0, 
\quad 
\delta\to0.
\label{c10}
\end{equation}
We have
\begin{equation}
\zeta(H_0)A(\delta)=
\zeta(H)\varphi_\delta(H)-\zeta(H_0)\varphi_\delta(H_0)
+
(\zeta(H_0)-\zeta(H))\varphi_\delta(H).
\label{c11}
\end{equation}
Consider separately the three terms in the r.h.s. of \eqref{c11}. 
Since $\varphi(x)\to0$ as $\abs{x}\to0$, we have
$\norm{\zeta(H)\varphi_\delta(H)}\to0$ and 
$\norm{\zeta(H_0)\varphi_\delta(H_0)}\to0$ as $\delta\to+0$. 
Next, denoting $\wt \zeta(x)=1-\zeta(x)$, we have 
$\wt \zeta\in C_0(\bbR)$. It follows that the operator
$$
\zeta(H_0)-\zeta(H)=\wt \zeta(H)-\wt \zeta(H_0)
$$
is compact. By our assumptions we have $0\in\Omega$  
($\Omega$ is defined in Proposition~\ref{pr1}), 
and therefore $0$ is not an eigenvalue of $H$. It follows that  
$\varphi_\delta(H)$ converges to zero strongly as $\delta\to0$. 
Thus, we get 
$$
\norm{\varphi_\delta(H)(\zeta(H_0)-\zeta(H))}\to0
$$
as $\delta\to0$, and therefore the last term 
in the r.h.s. of \eqref{c11} converges to zero 
in the operator norm.
Thus, \eqref{c2} is proven.

In the rest of this section, we prove \eqref{c1}. 

\subsection{Notation}\label{sec.c4}
Denote $Y(z)=V_0(I+T(z)V_0)^{-1}$, $\Im z>0$; 
by Proposition~\ref{pr1}, the limit $Y(x+i0)$ exists for 
all $x\in[-a,a]$ and is a H\"older continuous function of $x$
in the operator norm.
Recall that the operators $Z(x)$, $x\in[-a,a]$, are defined
in Section~\ref{sec.b1} (see \eqref{b2} -- \eqref{b4}).  
Let us define the operators
$$
\calZ,\calZ_0: L^2(\bbR,\calH)\to L^2(\bbR,\calN)
\text{ and }
\calY,\calY_0: L^2(\bbR,\calH)\to L^2(\bbR,\calH)
$$ 
by 
$$
(\calZ_0 f)(x)=Z(0)f(x), 
\quad
(\calY_0 f)(x)=Y(+i0)f(x),
\quad
x\in\bbR,
$$
and 
$$
(\calZ f)(x)=
\begin{cases}
Z(x)f(x), & \abs{x}\leq a, 
\\
0, & \abs{x}>a,
\end{cases}
\quad
(\calY f)(x)=
\begin{cases}
Y(x+i0)f(x), & \abs{x}\leq a, 
\\
0, & \abs{x}>a.
\end{cases}
$$
We will also need the operator valued versions of the Hardy projections $P_\pm$
(see \eqref{hp}). 
Denote by $\calP_{\calH,\pm}$ the operators in $L^2(\bbR,\calH)=L^2(\bbR)\otimes \calH$
defined by $\calP_{\calH,\pm}=P_\pm\otimes I_\calH$. 
Similarly, $\calP_{\calN,\pm}$ 
are the operators $P_\pm\otimes I_\calN$ 
in $L^2(\bbR,\calN)=L^2(\bbR)\otimes \calN$.

\subsection{Two formulas}\label{sec.c5}
Here we give formulas for the operators $QA_0(\delta) Q^*$ and
$E_{H_0}(-a,a) A(\delta) E_{H_0}(-a,a)$
in terms of the operators $\calZ,\calZ_0,\calY,\calY_0$ introduced above.

\begin{lemma}\label{lma.d1}
Let $\varphi$ be compactly supported. 
Then for all sufficiently small $\delta>0$, 
we have
\begin{align}
QA_0(\delta) Q^*
&=
4\pi\, \Im (\calF_a^*{\Pi_\calN} \calP_{\calN,-} 
\calZ_0\calY_0\calZ_0^*\boldsymbol{\varphi_\delta}
\calP_{\calN,+}\Pi_\calN^*\calF_a),
\label{d1}
\\
E_{H_0}(-a,a) A(\delta) E_{H_0}(-a,a)
&=
4\pi\, \Im (\calF_a^*{\Pi_\calN} \calZ \calP_{\calH,-}
\calY \boldsymbol{\varphi_\delta}
\calP_{\calH,+}\calZ^*\Pi_\calN^*\calF_a).
\label{d2}
\end{align}
\end{lemma}
\begin{proof}
1. Let us prove \eqref{d1}. 
By the stationary representation \eqref{b5a} for the scattering matrix, 
we have 
$$
S(0)=I-2\pi i Z(0)Y(+i0)Z(0)^*. 
$$
By the definition \eqref{c0} of $A_0(\delta)$, we get
\begin{multline*}
A_0(\delta)
=
2\, \Re(P_- \boldsymbol{\varphi_\delta}P_+\otimes (S(0)-I))
\\
=
4\pi\, \Im (P_- \boldsymbol{\varphi_\delta}P_+\otimes Z(0)Y(+i0)Z(0)^*)
=
4\pi\, \Im (\calP_{\calN,-} \calZ_0 \calY_0 \calZ_0^* \boldsymbol{\varphi_\delta}\calP_{\calN,+}),
\end{multline*}
and \eqref{d1} follows.

2. 
We will use the resolvent identity in the form
\begin{equation}
R(z)-R_0(z)
=
-(GR_0(\overline z))^* Y(z)GR_0(z), 
\quad \Im z>0.
\label{c12}
\end{equation}
Let us recall the derivation of \eqref{c12} (see e.g. \cite[Section~1.9]{Yafaev}). 
Iterating the usual resolvent identity, we get
\begin{align}
R(z)-R_0(z)
=
-R(z)VR_0(z)
&=
-R_0(z)VR_0(z)
+R_0(z)VR(z)VR_0(z)
\label{c12a}
\\
&=-R_0(z)GV_0(I-GR(z)GV_0)GR_0(z).
\label{c12b}
\end{align}
We also have the identity 
\begin{equation}
(I-GR(z)GV_0)(I+GR_0(z)GV_0)=I,
\label{c12c}
\end{equation}
which can be verified by expanding and using \eqref{c12a}. 
Substituting \eqref{c12c} into \eqref{c12b} and using the 
notation $Y(z)$, we obtain \eqref{c12}.

3.  Let us prove \eqref{d2}. 
Let $\delta>0$ be sufficiently small so that $\supp \varphi_\delta\subset [-a,a]$. 
First recall a version of Stone's formula:
$$
(\varphi_\delta(H) f,f)
=
\frac1\pi
\lim_{\eps\to+0}\Im \int_{-a}^a (R(x+i\eps)f,f)\varphi_\delta(x)dx,
$$
for any $f\in\calH$. 
Using this formula, a similar formula for $\varphi_\delta(H_0)$ 
and the resolvent identity \eqref{c12}, we get
\begin{equation}
(A(\delta)f,f)
=
-\frac1\pi\lim_{\eps\to+0}\Im \int_{-a}^a
(Y(x+i\eps)GR_0(x+i\eps)f,GR_0(x-i\eps)f)\varphi_\delta(x)dx,
\label{c14}
\end{equation}
for any $f\in\calH$. 

Next, let $f\in\Ran E_{H_0}(-a,a)$ and $F=\calF_a f$. 
From \eqref{b4a} we obtain for any  $\Im z\not=0$
\begin{equation}
GR_0(z)f
=
\int_{-a}^a \frac{Z(t)^*F(t)}{t-z} dt.
\label{c15}
\end{equation}
Combining \eqref{c14} and \eqref{c15}, we obtain
\begin{gather*}
(A(\delta)f,f)
=
4\pi \lim_{\eps\to+0}\Im \int_{-a}^a dx  \int_{-a}^a dt \int_{-a}^a ds 
\, (M(x,t,s)F(t),F(s)),
\\
M(x,t,s)
=
Z(s)\frac{1}{2\pi i}\frac{1}{s-x-i\eps}Y(x+i\eps)\varphi_\delta(x)
\left(-\frac{1}{2\pi i}\right)\frac{1}{x-t+i\eps}Z(t)^*.
\end{gather*}
Recalling formula \eqref{hp} for $P_\pm$,  
we obtain the required identity
\eqref{d2}. 
\end{proof}

\subsection{Compactness lemma}\label{sec.c6}

\begin{lemma}\label{lma.c4}
The operators 
\begin{equation}
L^\pm
=
{\Pi_\calN}(\calZ \calP_{\calH,\pm}
-
\calP_{\calN,\pm}\calZ)\Pi_\calH^*: 
L^2((-a,a),\calH)\to L^2((-a,a),\calN)
\label{d7}
\end{equation}
are compact. 
\end{lemma}
\begin{proof}
Clearly, it suffices to prove that for any $\eps>0$ 
the operator $L^\pm$ can be represented as 
$L^\pm=L^\pm_\eps+\wt L^\pm_\eps$, where
$L^\pm_\eps$ is compact and $\norm{\wt L^\pm_\eps}\leq \eps$.

1. Let us prove that for any $\eps>0$ one can find
an operator valued polynomial 
$Z_\eps(x)$  with coefficients being 
finite rank operators from $\calH$ to $\calN$ and such that 
$\norm{Z_\eps(x)-Z(x)}\leq \eps $ for all $x\in[-a,a]$. 
Let $P_{n,\calH}$ be a sequence of orthogonal 
projections in $\calH$ of finite rank which converges strongly 
to the identity operator $I_\calH$. Let $P_{n,\calN}$ be a similar
sequence for the space $\calN$. 
For each $x\in[-a,a]$, by compactness of $Z(x)$ we have
$$
\norm{Z(x)(P_{n,\calH}-I_\calH)}\to0, 
\quad
\norm{(P_{n,\calN}-I_\calN)Z(x)}\to0,
$$
and therefore, by the compactness of the interval $[-a,a]$, 
the above convergence holds true uniformly over $x\in[-a,a]$. 
It follows that for a sufficiently large $n$ we have 
$$
\norm{
P_{n,\calN}Z(x)P_{n,\calH}-Z(x)
}\leq \eps/2, 
\quad x\in[-a,a]. 
$$
The operator 
$P_{n,\calN}Z(x)P_{n,\calH}$ 
can be thought of as a matrix with respect to some bases
in $\Ran P_{n,\calN}$ and $\Ran P_{n,\calH}$; the elements of this matrix
are continuous functions in $x$. 
By the Weierstrass approximation theorem, 
the elements of this matrix can be approximated by polynomials
uniformly on $[-a,a]$. 
This yields the required approximation $Z_\eps(x)$ of $Z(x)$.  

2. 
For $Z_\eps$ as above, let us write $Z=Z_\eps+\wt Z_\eps$, 
with  $\norm{\wt Z_\eps(x)}\leq \eps$ for all $x\in[-a,a]$. 
This generates a decomposition $L^\pm= L^\pm_\eps+\wt L^\pm_\eps$
with 
$$
\norm{\wt L^\pm_\eps}
=
\norm{
{\Pi_\calN}(\wt \calZ_\eps \calP_{\calH,\pm}
-
\calP_{\calN,\pm}\wt \calZ_\eps)\Pi_\calH^*
}
\leq 
2 \norm{\wt \calZ_\eps}
\leq
2\eps.
$$ 
It suffices to prove that $L^\pm_\eps$ is compact. 
The operator $L^\pm_\eps$ is an integral operator
with the kernel 
$$
L^\pm_\eps(x,y)=\mp\frac1{2\pi i}\frac{Z_\eps(x)-Z_\eps(y)}{x-y}.
$$
This is a smooth matrix valued kernel, and therefore 
$L^\pm_\eps$ is compact. 
\end{proof}

\subsection{Proof of \eqref{c1}}\label{sec.c7}

1. First let us prove \eqref{c1} for compactly supported $\varphi$. 
By \eqref{c9a}, it suffices to prove that 
\begin{equation}
\norm{E_{H_0}(-a,a)A(\delta)E_{H_0}(-a,a)-QA_0(\delta)Q^*}
\to0, 
\quad \delta\to+0.
\label{d2a}
\end{equation}
Since $\varphi$ is compactly supported, we can apply 
Lemma~\ref{lma.d1} to represent  the operators in \eqref{d2a}
in terms of $\calZ$, $\calY$, etc. 
Thus, we see that \eqref{d2a} will follow from 
\begin{equation}
\norm{
{\Pi_\calN} \calZ\calP_{\calH,-}\calY\boldsymbol{\varphi_\delta}
\calP_{\calH,+}\calZ^*\Pi_\calN^*
-
{\Pi_\calN} \calP_{\calN,-} 
\calZ_0\calY_0\calZ_0^*
\boldsymbol{\varphi_\delta}
\calP_{\calN,+}\Pi_\calN^*
}\to0, \quad \delta\to+0.
\label{d3}
\end{equation}
Below we prove \eqref{d3}.

2. Note that $\boldsymbol{\varphi_\delta}$ 
converges to zero strongly as $\delta\to+0$. 
Next, if $\delta$ is sufficiently small 
so that $\supp\varphi_\delta\subset[-a,a]$, we have
$$
{\Pi_\calN}
(\calZ\calP_{\calH,-}
-
\calP_{\calN,-}\calZ)\boldsymbol{\varphi_\delta}
=
{\Pi_\calN}
(\calZ\calP_{\calH,-}
-
\calP_{\calN,-}\calZ)\Pi_\calH^*{\Pi_\calH}\boldsymbol{\varphi_\delta}
=
L^{-}{\Pi_\calH}\boldsymbol{\varphi_\delta},
$$
where $L^-$ is defined in \eqref{d7}. 
Thus by Lemma~\ref{lma.c4}, 
\begin{equation}
\norm{{\Pi_\calN}(\calZ\calP_{\calH,-} -\calP_{\calN,-}\calZ)\boldsymbol{\varphi_\delta}}
\to0, \quad \delta\to+0.
\label{d4}
\end{equation}
Writing a similar relation for $L^+$ instead of $L^-$ and 
taking adjoints, one obtains 
\begin{equation}
\norm{\boldsymbol{\varphi_\delta}(\calP_{\calH,+}\calZ^*-\calZ^*\calP_{\calN,+})\Pi_\calN^*}
\to0, \quad \delta\to+0.
\label{d5}
\end{equation}
Combining \eqref{d4} and \eqref{d5} and using
the commutation  
$\boldsymbol{\varphi_\delta}\calY=\calY\boldsymbol{\varphi_\delta}$, 
we obtain
\begin{equation}
\norm{
{\Pi_\calN} \calZ\calP_{\calH,-}\calY\boldsymbol{\varphi_\delta}
\calP_{\calH,+}\calZ^*\Pi_\calN^*
-
{\Pi_\calN} \calP_{\calN,-}\calZ\calY
\boldsymbol{\varphi_\delta}\calZ^*\calP_{\calN,+}\Pi_\calN^*}
\to0, 
\quad \delta\to+0.
\label{d6}
\end{equation}
Recall that by Proposition~\ref{pr1} the operator $Y(x+i0)$ is continuous
in $x\in[-a,a]$. Using this fact and the 
continuity of $Z(x)$ in $x$ we get
$$
\norm{
\calZ\calY\boldsymbol{\varphi_\delta}\calZ^*
-
\calZ_0\calY_0\boldsymbol{\varphi_\delta}\calZ_0^*
}
\to0, 
\quad \delta\to+0.
$$
Combining the last relation with \eqref{d6} 
and using the commutation 
$\boldsymbol{\varphi_\delta}\calZ_0=\calZ_0\boldsymbol{\varphi_\delta}$,
we arrive at \eqref{d3}.

3. 
It remains to extend \eqref{c1} from compactly 
supported functions to general $\varphi\in C_0(\bbR)$. 
For $\varphi\in C_0(\bbR)$ and any given $\eps>0$ 
let us approximate $\varphi$ by 
a compactly supported function $\wt \varphi$ such that 
$\norm{\varphi-\wt\varphi}_{C(\bbR)}\leq \eps$. 
Let $\wt A(\delta)$, $\wt A_0(\delta)$ be the operators
$A(\delta)$, $A_0(\delta)$ corresponding to $\wt \varphi$. 
Then by steps 1,2 above we have $\norm{\wt A(\delta)-Q \wt A_0(\delta) Q^*}\to0$
as $\delta\to+0$. 
Next, directly from the definition of $A(\delta)$ and $A_0(\delta)$ we obtain 
\begin{align*}
\norm{A(\delta)-\wt A(\delta)}&\leq 2\norm{\varphi-\wt \varphi}_{C(\bbR)}\leq 2\eps,
\\
\norm{A_0(\delta)-\wt A_0(\delta)}&\leq 4\norm{\varphi-\wt \varphi}_{C(\bbR)}\leq 4\eps.
\end{align*}
It follows that
$$
\limsup_{\delta\to+0}\norm{A(\delta)-Q A_0(\delta) Q^*}\leq 6\eps.
$$
Since $\eps>0$ is arbitrary, we get \eqref{c1}.

\qed


\begin{thebibliography}{00}

\bibitem{Agmon}
{\sc S.~Agmon,} 
\emph{Spectral properties of Schr\"odinger operators and scattering theory},
Ann. Scuola Norm. Sup. Pisa   (IV) \textbf{2} (1975), no. 4, 151--218.


\bibitem{BK}
{\sc M.~Sh.~Birman, M.~G.~Kre\u{\i}n,} 
\emph{On the theory of wave operators and
scattering operators,}   Soviet Math. Dokl. \textbf{3} (1962), 740--744.

\bibitem{BuslaevFaddeev}
{\sc V.~S.~Buslaev, L.~D.~Faddeev,}
\emph{Formulas for traces for a singular {S}turm-{L}iouville
differential operator},
Soviet Math. Dokl.
\textbf{1} (1960), {451--454}.

\bibitem{Buslaev}
{\sc V.~S.~Buslaev},
\emph{The trace formulae and certain asymptotic estimates of the
              kernel of the resolvent for the {S}chr\"odinger operator in
              three-dimensional space},
{Probl. {M}ath. {P}hys., {N}o. {I}, {S}pectral {T}heory and
              {W}ave {P}rocesses ({R}ussian)},
{Izdat. Leningrad. Univ., Leningrad},
1966, pp. 82--101.



\bibitem{Kato}
{\sc T.~Kato,} 
\emph{Growth properties of solutions of the reduced wave equation with a variable coefficient,}
Comm. Pure Appl. Math. \textbf{12} (1959), 403--425.


\bibitem{Kato3}
{\sc T.~Kato,} 
\emph{Some results on potential scattering,}
Proc. Intern. Conference on Funct. Anal. and Related Topics, 
Univ. of Tokyo Press, Tokyo (1969), 206--215.


\bibitem{KM}
{\sc V.~Kostrykin, K.~Makarov,} 
\emph{On Krein's example,}
Proc. Amer. Math. Soc. \textbf{136} (2008), no. 6, 2067--2071.



\bibitem{Krein}
{\sc M.~G.~Kre\u{\i}n,} 
\emph{On the trace formula in perturbation theory (Russian)},
Mat. Sb. \textbf{33 (75)}  (1953), no. 3, 597--626.


\bibitem{Kuroda}
{\sc S.~T.~Kuroda,} 
\emph{Scattering theory for differential operators,}
J. Math. Soc. Japan \textbf{25} (1973),
\emph{I Operator theory,} 75--104,
\emph{II Self-adjoint elliptic operators,}
222--234.


\bibitem{Lifshits}
{\sc I.~M.~Lif{\v{s}}ic,}
 \emph{On degenerate regular perturbations. {II}. {Q}uasicontinuous
and continuous spectrum},
Akad. Nauk SSSR. Zhurnal Eksper. Teoret. Fiz.
\textbf{17}
(1947), 1076--1089.


\bibitem{Peller1}
{\sc V.~Peller,} 
\emph{Hankel operators in perturbation theory of unitary and self-adjoint operators}, 
Funct. Anal. Appl. {\bf 19}  (1985), 111--123.


\bibitem{Peller}
{\sc V.~Peller,} 
\emph{Hankel operators and their applications. }
Springer-Verlag, New York, 2003.

\bibitem{Push}
{\sc A.~Pushnitski,}
\emph{The scattering matrix and the differences of spectral projections},
Bulletin London Math. Soc. \textbf{40}  (2008), 227--238.


\bibitem{PushYaf}
{\sc A.~Pushnitski, D.~Yafaev,}
\emph{Spectral theory of 
discontinuous functions of self-adjoint operators  
and scattering theory,}
J. Functional Analysis \textbf{259} (2010), 1950--1973.

\bibitem{RS2}
{\sc M.~Reed, B.~Simon,}
\emph{Methods of modern mathematical physics. II. Fourier Analysis, Self-adjointness.}
Academic Press, 1975.


\bibitem{Yafaev}
{\sc D.~R.~Yafaev,}
\emph{Mathematical scattering theory. General theory.}
American Mathematical Society, Providence, RI, 1992. 


\end{thebibliography}
\end{document}